\documentclass[a4paper,oneside,11pt, reqno]{amsart}
\usepackage{graphicx}
\usepackage{float}
\usepackage{pstricks}
\usepackage{amscd} 
\usepackage{amsmath}
\usepackage{amsxtra}
\usepackage{tikz}
\usepackage{tikz-cd}
\usepackage[T1]{fontenc}
\usepackage[utf8]{inputenc}
\usepackage{lipsum}
\usepackage{tikz}
\usetikzlibrary{arrows}
\usetikzlibrary{calc}
\usepackage{hyperref}

\usepackage{amsfonts,amssymb,amsmath,amsthm}
\usepackage{url}
\usepackage{enumerate}
\theoremstyle{plain}
\newtheorem{theorem}{Theorem}[section]

\newtheorem{lemma}[theorem]{Lemma}
\newtheorem{proposition}[theorem]{Proposition}
\newtheorem{corollary}[theorem]{Corollary}

\theoremstyle{definition}
\newtheorem{definition}[theorem]{Definition}

\newtheorem{example}[theorem]{Example}

\theoremstyle{remark}
\newtheorem{remark}[theorem]{Remark}

\numberwithin{equation}{section}

\newlength{\struh}
\newlength{\textminustop}
\usepackage{mathrsfs}
\usepackage{epsfig}
\usepackage[active]{srcltx}

\newcommand{\ncom}{\newcommand}
\ncom{\bq}{\begin{equation}}
\ncom{\eq}{\end{equation}}
\ncom{\beqn}{\begin{eqnarray*}}
\ncom{\eeqn}{\end{eqnarray*}}
\ncom{\beq}{\begin{eqnarray}}
\ncom{\eeq}{\end{eqnarray}}
\ncom{\nno}{\nonumber}
\ncom{\rar}{\rightarrow}
\ncom{\Rar}{\Rightarrow}
\ncom{\noin}{\noindent}
\ncom{\bc}{\begin{centre}}
\ncom{\ec}{\end{centre}}
\ncom{\sz}{\scriptsize}
\ncom{\rf}{\ref}
\ncom{\sgm}{\sigma}
\ncom{\Sgm}{\Sigma}
\ncom{\dt}{\delta}
\ncom{\Dt}{Delta}
\ncom{\lmd}{\lambda}
\ncom{\Lmd}{\Lambda}
\ncom{\eps}{\epsilon}
\ncom{\pcc}{\stackrel{P}{>}}
\ncom{\dist}{{\rm\,dist}}
\usepackage{cite}
\ncom{\sspan}{{\rm\,span}}
\ncom{\im}{{\rm Im\,}}
\ncom{\sgn}{{\rm sgn\,}}
\ncom{\ba}{\begin{array}}
\ncom{\ea}{\end{array}}
\ncom{\eop}{\hfill{{\rule{2.5mm}{2.5mm}}}}
\ncom{\eoe}{\hfill{{\rule{1.5mm}{1.5mm}}}}
\ncom{\eof}{\hfill{{\rule{1.5mm}{1.5mm}}}}
\ncom{\hone}{\mbox{\hspace{1em}}}
\ncom{\htwo}{\mbox{\hspace{2em}}}
\ncom{\hthree}{\mbox{\hspace{3em}}}
\ncom{\hfour}{\mbox{\hspace{4em}}}
\ncom{\hsev}{\mbox{\hspace{7em}}}
\ncom{\vone}{\vskip 2ex}
\ncom{\vtwo}{\vskip 4ex}
\ncom{\vonee}{\vskip 1.5ex}
\ncom{\vthree}{\vskip 6ex}
\ncom{\vfour}{\vspace*{8ex}}
\ncom{\norm}{\|\;\;\|}
\ncom{\integ}[4]{\int_{#1}^{#2}\,{#3}\,d{#4}}
\ncom{\inp}[2]{\langle{#1},\,{#2} \rangle}
\ncom{\Inp}[2]{\Langle{#1},\,{#2} \Langle}
\ncom{\vspan}[1]{{{\rm\,span}\#1 \}}}
\ncom{\dm}[1]{\displaystyle {#1}}



\begin{document}
\title[The Hartogs triangle core operator]{pairs of Commuting Isometries via new Core Operator}

\author[S. Jain]{Shubham Jain}
\address [S. Jain]{Indian Statistical Institute
Statistics and Mathematics Unit,
8th Mile, Mysore Road,
RVCE Post,
Bangalore, 560059,
Karnataka, India}
\email{shubhamjain\_pd@isibang.ac.in, shubhamjain4343@gmail.com} 

\begin{abstract}
We study commuting isometric pairs $(V_1,V_2)$ that admit a factorization of the form $(V_2V_3,\,V_2)$ for some isometry $V_3$. To investigate this class, we introduce the notion of core operator associated with a commuting isometric pair $(V_1, V_2):$
\begin{equation*}
H(V_1, V_2):=V_2V_2^*-V_1V_1^*-V_2^2V_2^{*2}+V_1V_2V_1^*V_2^*.
\end{equation*}
This core operator is closely associated with the domain Hartogs triangle. In particular, we establish a model for pure commuting isometric pairs satisfying
$H(V_1,V_2)\geq 0$ via the coordinate multiplication operators on a
vector-valued Hardy space over the Hartogs triangle. We also investigate the properties of the core operator and its applications to commuting isometric pairs.
\end{abstract}
\maketitle
\section{Introduction}
A fundamental problem in Hilbert space operator theory is the classification and representation of a tuple of commuting isometries. For a single isometry, the solution is complete and explicit. Suppose $V$ is an isometry acting on a complex separable Hilbert space $\mathcal{H}$. Then there exist another Hilbert space $\mathcal{H}_u$ and a unitary operator $U$ on $\mathcal{H}_u$ such that $V$ is unitarily equivalent to the operator
$$
\begin{bmatrix}
S \otimes I_{\mathcal{W}} & 0 \\
0 & U
\end{bmatrix}
$$
acting on $(\ell^2(\mathbb{Z}_+) \otimes \mathcal{W}) \oplus \mathcal{H}_u$. Here, $\mathcal{W} = \operatorname{ker} V^*$ denotes the wandering subspace associated with $V$, and $S$ represents the unilateral shift on $\ell^2(\mathbb{Z}_+)$. This decomposition is a classical result attributed to von Neumann and Wold (see \cite{H1961, N1930, Wold}). However, the structure of a pair of commuting isometries is still not fully known. In the past few decades, many papers were devoted to this subject (see, e.g. \cite{BCG2024, BRV2022, AKM2012, BDF2006, Bu2019, BKPS2017, BKS2015,GG2004, KO1999, MSS2019, Popovici I, Popovici, Slocinski} and the references therein).

In the Hardy space over the polydisc and Drury-Arveson space on the unit ball, the core operator associated with these domains played an essential role in studying the submodules (see \cite{GY2004, HQY2015, MSS2019, BRV2022, BCG2024, DPSS2024, G2004, A1998} and references therein). In the bidisc setting, Guo and Yang \cite{GY2004} introduced the bidisc core operator as an invariant associated with submodules, and subsequently He, Qin, and Yang \cite{HQY2015} showed that it is a useful tool for studying the pairs of commuting isometries. More recently, bidisc core operators have also been used to investigate spectral properties of commuting isometries and related operator-theoretic questions (see, e.g. \cite{BRV2022, BCG2024} and references therein). The \textit{defect operator or  bidisc core operator} $C(V_1, V_2)$ of a pair $(V_1, V_2)$ of commuting isometries is defined as 
\beq \label{bidisc core operator}
C(V_1, V_2):=I-V_1V_1^*-V_2V_2^*+V_1V_2V_1^*V_2^*.
\eeq 
Let $(V_1,V_2)$ be a pair of commuting isometries on a Hilbert space $\mathcal H$. By Douglas's range inclusion theorem \cite{Douglas}, there exists an isometry $V_3$ such that
\begin{align}\label{V_1=V_2V_3}
    (V_1,V_2)=(V_2V_3,V_2) \mbox{ if and only if } \operatorname{ran}V_1\subseteq \operatorname{ran}V_2.
\end{align}
A natural example of such an isometric pair is given by the pair of multiplication operators by the coordinate functions on the Hardy space of the Hartogs triangle (see \cite[Equation~(4.11)]{CJP2023}). To establish an equivalent criterion for the existence of such a factorization (see \eqref{V_1=V_2V_3}), we introduce and study the core operator associated with the Hartogs triangle. Recall that the domain Hartogs triangle is given by $$\triangle_H=\{(z,w)\in \mathbb C^2:|z|<|w|<1\}.$$
\begin{definition}
For a pair $(V_1, V_2)$ of commuting isometries on a Hilbert space $\mathcal{H},$ we define the {\it core operator associated with the Hartogs triangle} $H(V_1, V_2)$ of $(V_1, V_2)$ as 
\beq \label{Hartogs triangle core operator}
H(V_1, V_2):=V_2V_2^*-V_1V_1^*-V_2^2V_2^{*2}+V_1V_2V_1^*V_2^*.
\eeq
\end{definition}

Note that $H(V_1, V_2)$ is a self-adjoint operator on $\mathcal{H}.$ Also, $H(V_1, V_2)=0$ on $V_1V_2^2\mathcal{H}.$ Also, note that $V_2V_2^*-V_1V_1^*$ may not be a projection. In fact,
\beqn
V_2V_2^*-V_1V_1^*=(I-V_1V_1^*)-(I- V_2V_2^*)=P_{\mathcal{W}(V_1)}-P_{\mathcal{W}(V_2)},
\eeqn
which is not a projection unless $$P_{\mathcal{W}(V_1)}P_{\mathcal{W}(V_2)}=P_{\mathcal{W}(V_2)}P_{\mathcal{W}(V_1)}=P_{\mathcal{W}(V_2)},$$
where $\mathcal{W}(V):=\mathcal{H} \ominus V\mathcal{H}$ denote the wandering subspace (see \cite{H1961}) for an isometry $V$ on a Hilbert space $\mathcal{H}$.
\begin{example}
    Let $M_i$ denote the operator of multiplication by the coordinate function $z_i$ for $i=1,2$ in the Hardy space $H^2(\mathbb D^2)$. Then $(M_1, M_2)$ is a pair of commuting isometries; however, the operator $M_2M_2^* - M_1M_1^*$ is not a projection. 
\end{example}
Note that $(z_1, z_2)\in \mathbb D^2$ is in $\triangle_H$ if and only if $|z_2|^2-|z_1|^2-|z_2|^4+|z_1z_2|^2>0$. From now onward, we shall refer to the operator defined in \eqref{Hartogs triangle core operator} as the \emph{Hartogs triangle core operator}, or simply the \emph{$\triangle_H$-core operator}.

Let $\mathbb N$ and $\mathbb{Z}_+$ denote the set of natural numbers and nonnegative integers, respectively. The space of bounded linear operators on a Hilbert space $\mathcal H$ is denoted by $B(\mathcal{H}).$ For a closed subspace $\mathcal{M}$ of $\mathcal{H},$ let $P_{\mathcal{M}}$ denote the orthogonal projection onto $\mathcal{M}.$ For a bounded linear operator $T \in B(\mathcal{H})$ and a closed subspace $\mathcal{M}$ of $\mathcal{H},$ we say $\mathcal{M}$ is invariant under $T$ if $T(\mathcal{M}) \subseteq \mathcal{M}$ and $\mathcal{M}$ reduces $T$ if $T(\mathcal{M}) \subseteq \mathcal{M}$ and $T^*(\mathcal{M}) \subseteq \mathcal{M}.$ For a bounded linear operator $T,$ $\operatorname{ran}T$ and $\operatorname{ker}T$ stands for range and kernel of $T.$ For a finite-dimensional subspace $\mathcal M$ of $\mathcal H,$ $\operatorname{dim} \mathcal M$ denotes the dimension of $\mathcal M$. Recall that an isometry $V$ on a Hilbert space $\mathcal H$ is said to be pure if $\displaystyle \lim_{n \rightarrow\infty}V^{*n}=0$ in the strong operator topology.

The outline of the paper is as follows. In Section \ref{S3}, we begin by proving two preliminary results (Lemmas \ref{wandering subspaces relation} and \ref{C-H relation}). We then establish the following result, which gives the relationship between the $\triangle_H$-core operator and the bidisc core operator.
\begin{theorem} \label{Hartogs bidisc core relation}
    Let $(V_1, V_2)$ be a pair of commuting isometries on a Hilbert space $\mathcal H.$ Then the following are equivalent:
    
    \begin{itemize}
       \item[$\mathrm{(i)}$]  $H(V_1, V_2) \geq 0.$ 
        \item[$\mathrm{(ii)}$] There exists an isometry $V_3$ commuting with $V_2$ such that $V_1=V_2V_3$ and $C(V_3, V_2) \geq 0.$
        \item[$\mathrm{(iii)}$] $V_1(\mathcal{W}(V_2)) \subseteq V_2(\mathcal{W}(V_2)).$
    \end{itemize} 
\end{theorem}
As an application, we obtain the following model theorem for commuting pairs of pure isometries.
\begin{proposition} \label{modelthm}
     Let $(V_1, V_2)$ be a pair of pure commuting isometries on a Hilbert space $\mathcal H$ such that $H(V_1, V_2) \geq 0.$ Then $(V_1, V_2)$ is unitarily equivalent to $(\mathscr M_{z_1}^\mathcal Z, \mathscr M_{z_2}^\mathcal Z)$ on $H^2_{\mathcal Z}(\triangle_H)$. 
\end{proposition}
See \eqref{vectorhardy} for the definition of $H^2_{\mathcal Z}(\triangle_H)$.
We further prove that, if $\operatorname{ker}(V_i^*)$ is finite-dimensional for $i=1,2$, then $H(V_1,V_2)\geq 0$ or $H(V_1,V_2)\leq 0$ holds if and only if $H(V_1,V_2)=0$ (see Corollary \ref{wanderingfinite}). We also obtain a characterization of doubly commuting pairs of isometries in terms of the Hartogs triangle core operator (see Proposition \ref{doublychara}). We conclude the section by presenting several equivalent characterizations for $V_2$ to be unitary in the pair $(V_1,V_2)$ (see Theorem \ref{mainthmiso}).

A commuting isometric pair $(V_1, V_2)$ on a Hilbert space $\mathcal H$ is said to be completely non-unitary if the isometry $V_1V_2$ is pure, that is,
\beqn
\bigcap_{n=0}^{\infty}(V_1V_2)^k \mathcal H=\{0\}.
\eeqn
We now recall the model theorem for completely non-unitary commuting isometric pairs from \cite{BCL1978}. Let $\mathcal S$ be a separable Hilbert space and let $H^2_{\mathcal S}(\mathbb T)$ be the $\mathcal S$-valued Hardy space, that is,
\beqn
H^2_{\mathcal S}(\mathbb T):=\left\{f(z)=\sum_{j=0}^{\infty}a_jz^j : a_j \in \mathcal S, |z|=1, \|f\|^2=\sum_{j=0}^{\infty}\|a_j\|^2_{\mathcal S}<\infty \right\}.
\eeqn
Let $M_z$ denote the multiplication operator by the coordinate function $z$ in $H^2_\mathcal S(\mathbb T).$ It is easy to see that $M_z$ is a unilateral shift of multiplicity equal to the dimension of $\mathcal S.$ It is easy to see that the space $H^2_\mathcal S(\mathbb T)$ is identified with tensor product space $H^2(\mathbb T) \otimes \mathcal S.$ Let $U$ and $P$ be a unitary and an orthogonal projection on $\mathcal S,$ respectively. Then the BCL triple $(\mathcal S, U, P)$ determines a pair of isometries $(M_{\phi_1}, M_{\phi_2})$ on $H^2_\mathcal S(\mathbb T)$ by
\begin{align} \label{BCLpairss2}
 &M_{\phi_1}
=
I_{H^2(\mathbb T)}\otimes PU^*
+
S_z\otimes P^\perp U^*, \notag\\&M_{\phi_2}
=
I_{H^2(\mathbb T)}\otimes UP^\perp
+
S_z\otimes UP, 
\end{align}
where $S_z$ denotes the multiplication by the coordinate function $z$ in $H^2(\mathbb T).$ It follows that $M_{\phi_1}M_{\phi_2}=M_{\phi_2}M_{\phi_1}=M_z$ and hence $(M_{\phi_1}, M_{\phi_2})$ is completely non-unitary. Conversely, any completely non-unitary isometric pair can be realized, up to unitary equivalence, in the above form. Let $P_\mathbb C$ denote the orthogonal projection onto the subspace of constant functions in $H^2(\mathbb T).$ From now on, we refer $(M_{\phi_1}, M_{\phi_2})$ as BCL pair and $(\mathcal S, U, P)$ as BCL triple.

In Section \ref{S4}, we investigate completely non-unitary commuting isometric pairs through the Berger-Coburn-Lebow theorem. We first show that $H(M_{\phi_1},M_{\phi_2})=0$ if and only if $PUP=0$ and $U^2PU^{*2}=P$ (see Theorem \ref{pure=0}). We then establish criteria for the positive and negative definiteness of $H(M_{\phi_1},M_{\phi_2})$ in terms of the operators introduced in \eqref{ABCD} (see Propositions \ref{H>0} and \ref{H<0}). Next, we characterize the compactness of $H(M_{\phi_1},M_{\phi_2})$, showing that it is compact if and only if both $PUP$ and $U^2PU^{*2}-P$ are compact (see Theorem \ref{Compactchara}). Finally, we show that the nonzero part of compact contractive core operator $H(M_{\phi_1},M_{\phi_2})$ is unitarily equivalent to 
$$
\begin{pmatrix}
I_{1} & 0 & 0 & 0\\
0 & D & 0 & 0\\
0 & 0 & -I_{2} & 0\\
0 & 0 & 0 & -D
\end{pmatrix},
$$
where $I_1$ and $I_2$ are the identity operators and $D$ is a positive contractive diagonal operator (see Theorem \ref{structurethm}).
\section{Definiteness of Hartogs triangle core operator} \label{S3}
This section is devoted to the study of isometric pairs $(V_1, V_2)$ satisfying $H(V_1, V_2) \geq 0 $ or $H(V_1, V_2) \leq 0$, and investigate its relationship with the bidisc core operator. For a pair $(V_1, V_2)$ of commuting isometries on a Hilbert space $\mathcal{H},$ let $V_{m, n}$ denote the isometry $V_1^mV_2^n$ for $m,n \in \mathbb Z_+$. Let $\mathcal{W}_{m,n}$ denote the wandering subspace $\mathcal{W}(V_{m, n})$ for $V_{m, n}.$ The following simple observations about wandering subspaces for commuting isometries are helpful (see for $m=n=1,$ \cite[p.~6]{HQY2015}, \cite[Lemma~3.1]{MSS2019}).
\begin{lemma} \label{wandering subspaces relation}
    Let $(V_1, V_2)$ be a pair of commuting isometries on a Hilbert space $\mathcal{H}$ and $m, n \in \mathbb N.$ Then 
    \begin{itemize}
\item[$\mathrm{(i)}$] 
          $\mathcal{W}_{1,2}=\mathcal{W}_{0, 1} \oplus V_2\mathcal{W}_{1, 1}.$

          \item[$\mathrm{(ii)}$] $\mathcal{W}_{m,n}=\mathcal{W}_{m-1,n} \oplus V_{m-1, n}\mathcal{W}_{1,0}.$ 

           \item[$\mathrm{(iii)}$] $\mathcal{W}_{m,n}=\mathcal{W}_{m,n-1} \oplus V_{m, n-1}\mathcal{W}_{0,1}.$
    \end{itemize}
\end{lemma}
\begin{proof}
    $\mathrm{(i)}:$ Note that $\mathcal{W}_{0, 1} \cap V_2\mathcal{W}_{1, 1} =\{0\}.$ Let $x_1 \in \mathcal{W}_{0, 1}$ and $ x_2 \in V_2\mathcal{W}_{1, 1}.$ For $x_2 \in  V_2\mathcal{W}_{1, 1},$ there exist $x_2' \in \mathcal{W}_{1, 1}$ such that $V_2x_2'=x_2.$
    Now, $$V_1^*V_2^{*2}(x_1+x_2)=V_1^*V_2^{*2}x_1+V_1^*V_2^{*2}V_2x_2'=0.$$
    Hence, $\mathcal{W}_{0, 1} \oplus V_2\mathcal{W}_{1, 1}\subseteq \mathcal{W}_{1,2}.$ Conversely, let $y \in \mathcal{W}_{1,2}.$ Write $y=y-V_2V_2^*y+V_2V_2^*y.$ Clearly, $y-V_2V_2^*y \in \mathcal{W}_{0, 1}$ and $V_2^*y \in \mathcal{W}_{1, 1}.$ This completes the proof. 
    
    $\mathrm{(ii)}:$ Note that $\mathcal{W}_{m-1, n} \cap V_{m-1, n}\mathcal{W}_{1, 0} =\{0\}.$ Let $x_1\in \mathcal{W}_{m-1, n}$ and $x_2' \in \mathcal{W}_{1, 0}.$ Now, 
    \beqn
    V_{m, n}^*(x_1+V_{m-1, n}x_2') &=& V_1^{*m}V_2^{*n}x_1+V_1^{*m}V_2^{*n}V_{m-1, n}x_2' \\ &=& V_1^*V_1^{*m-1}V_2^{*n}x_1+V_1x_2'=0.
\eeqn
    Hence, $\mathcal{W}_{m-1, n} \oplus V_{m-1, n}\mathcal{W}_{1, 0}  \subseteq \mathcal{W}_{m, n}.$ Conversely, let $y \in \mathcal{W}_{m, n}.$ Write $y=y-V_{m-1, n}V_{m-1, n}^*y+V_{m-1, n}V_{m-1, n}^*y.$ Clearly, $y-V_{m-1, n}V_{m-1, n}^*y \in \mathcal{W}_{m-1, n}$ and $V_{m-1, n}^*y \in \mathcal{W}_{1, 0}.$ Similarly, $\mathrm{(iii)}$ follows.
\end{proof}
The following result is frequently used in this note.
\begin{lemma} \label{C-H relation}
    Let $(V_1, V_2)$ be a pair of commuting isometries on a Hilbert space $\mathcal H.$ Then
\begin{itemize}
    \item [$\mathrm{(i)}$] $H(V_1, V_2)-C(V_1, V_2)=-P_{\mathcal{W}_{0, 1}} +P_{V_2\mathcal{W}_{0, 1}}.$
    \item [$\mathrm{(ii)}$] $H(V_1, V_2)=P_{V_2\mathcal{W}_{0, 1}}-P_{V_1\mathcal{W}_{0, 1}}.$
\end{itemize} 
\end{lemma}
\begin{proof}
    $\mathrm{(i)}:$ It follows from \eqref{bidisc core operator} and \eqref{Hartogs triangle core operator} that \beqn
H(V_1, V_2)-C(V_1, V_2)&=&-I+2V_2V_2^*-V_2^2V_2^{*2} \notag \\ &=& -(I-V_2V_2^*)+V_2(I-V_2V_2^*)V_2^* \notag \\ &=& -P_{\mathcal{W}_{0, 1}} +P_{V_2\mathcal{W}_{0, 1}}.
    \eeqn
$\mathrm{(ii)}:$  It follows from \eqref{Hartogs triangle core operator} that \beqn  H(V_1, V_2) &=& V_2V_2^*-V_2^2V_2^{*2}-V_1V_1^*+V_1V_2V_1^*V_2^* \\ &=& V_2(I-V_2V_2^*)V_2^*-V_1(I-V_2V_2^*)V_1^* \\&=& P_{V_2\mathcal{W}_{0, 1}}-P_{V_1\mathcal{W}_{0, 1}}.
  \eeqn

\end{proof}
\begin{remark} \label{H-C relation}
 Let $\mathcal{W}_{0, 1} \neq \{0\}.$ Since closed subspaces $\mathcal{W}_{0,1}$ and $V_2\mathcal{W}_{0, 1}$ are orthogonal to each other, in view of Lemma \ref{C-H relation}, there exist $h$ and $k$ in $\mathcal{H}$ such that 
 \begin{align*}
    \langle (H(V_1, V_2)-C(V_1, V_2))h, h \rangle >0 \mbox{~and~} \langle (H(V_1, V_2)-C(V_1, V_2))k, k \rangle <0. 
 \end{align*}
 Also, it is easy to see that $H(V_1, V_2)-C(V_1, V_2)$ is a projection or negative of a projection if and only if $V_2$ is unitary.
\end{remark}
Let $(V_1, V_2)$ be a pair of commuting isometries on a Hilbert space $\mathcal{H}$ such that $H(V_1, V_2)=0.$ It follows from \eqref{Hartogs triangle core operator} that 
\begin{align} \label{onekerinano}
    \operatorname{ker}V_2^* \subseteq \operatorname{ker}V_1^*.
\end{align} It is worth noting that the converse may not be true, and the inclusion may be proper. Indeed, $(\mathscr M_{z_1}, \mathscr M_{z_2})$ is a pair of commuting isometries on the Hardy space $H^2(\triangle_H)$ of the Hartogs triangle for which $H(\mathscr M_{z_1}, \mathscr M_{z_2}) \neq 0$ and $\operatorname{ker}\mathscr M_{z_2}^*\subsetneq\operatorname{ker}\mathscr M_{z_1}^*,$ where $\mathscr M_{z_1}$ and $\mathscr M_{z_2}$ denote the bounded linear operators of multiplication by the coordinate functions $z_1$ and $z_2,$ respectively (see \cite[Remark~4.4]{CJP2023}). We now prove that \eqref{onekerinano} or \eqref{V_1=V_2V_3} holds whenever $H(V_1,V_2)\geq 0$ or $H(V_1,V_2)\leq 0$. As a consequence, we establish its connection with the bidisc core operator.

\begin{proof}[Proof of Theorem \ref{Hartogs bidisc core relation}]
$\mathrm{(i)} \iff \mathrm{(ii)} : $ Assume that $H(V_1, V_2) \geq 0.$ It suffices to check the behaviour of $H(V_1, V_2)$ on $\mathcal{W}_{1,2}$ as $H(V_1, V_2)=0$ on $V_1V_2^2\mathcal{H}.$ From Lemma \ref{wandering subspaces relation}($\mathrm{(i)} ~\&~ \mathrm{(ii)}$), it follows that
    \begin{align} \label{kernelw12}
\mathcal{W}_{1,2}=\mathcal{W}_{0,1} \oplus V_2\mathcal{W}_{0,1}\oplus V_2^2\mathcal{W}_{1,0}.
    \end{align}
For $x \in \mathcal{W}_{0,1},$ it is easy to see that $\langle H(V_1, V_2)x, x \rangle = -\|V_1^*x\|^2 \leq 0.$ Since $H(V_1, V_2) \geq 0,$ we have $V_1^*x=0.$ Hence, $\mathcal{W}_{0, 1} \subseteq \mathcal{W}_{1, 0}.$ As an isometry has a closed range, it implies $\operatorname{ran}V_1 \subseteq \operatorname{ran}V_2.$ For any $h \in \mathcal{H},$ there exists a unique $h' \in \mathcal{H}$ such that $V_1h=V_2h'.$ Define an operator $V_3 : \mathcal{H} \rightarrow \mathcal{H}$ such that $V_3h=h'$ (see \cite{Douglas}).

It is easy to see that $V_3$ is linear and $V_1=V_2V_3.$ Note that $$\|V_3h\| = \|h'\| = \|V_2h'\| = \|V_1h\| = \|h\|.$$ Hence, $V_3$ is bounded and, in fact, an isometry. Now, $$V_3V_2=V_2^*V_1V_2=V_1=V_2V_3.$$ Thus, $V_3$ commutes with $V_2.$ Since $C(V_3, V_2)=V_2^*H(V_1, V_2)V_2,$ we deduce that $C(V_3, V_2) \geq 0.$  For the converse part, observe that $H(V_1, V_2) = V_2C(V_3, V_2)V_2^*.$ Hence,  $H(V_1, V_2) \geq 0.$

$\mathrm{(i)} \iff \mathrm{(iii)} :$ It follows from the Lemma \ref{C-H relation}$\mathrm{(ii)}$.
\end{proof}
In view of Theorem \ref{Hartogs bidisc core relation}, $H(V_1, V_2)\geq 0$ yields a pair of doubly commuting isometries $(V_3, V_2)$ (see \cite[Theorem~3.4]{HQY2015}, \cite[Lemma 6.2]{MSS2019}). However, in this situation, the pair $(V_1, V_2)$ is never doubly commuting unless $V_2$ is unitary.
\begin{example}
 Consider an operator $V_1$ on $H^2(\mathbb D^2)$ defined as $V_1=M_1M_2.$ Then $(V_1, M_2)$ is pair of commuting isometries with $H(V_1, M_2) \geq0.$ In this situation, $V_3=M_1.$ Hence the pair $(M_1, M_2)$ is doubly commuting, however $(M_1M_2, M_2)$ is not.   
\end{example}
The Hardy space
$H^2(\triangle_H)$ of the Hartogs triangle is the reproducing kernel Hilbert space consisting of holomorphic
functions on $\triangle_H$ endowed with the norm
\begin{equation*}
\|f\|_{H^2(\triangle_H)}^2
=
\sup_{s,t\in (0,1)}
\frac{1}{4\pi^2}
\int_0^{2\pi}\int_0^{2\pi}
|f(ste^{i\theta},te^{i\gamma})|^2
\,st^2\,d\theta\,d\gamma,
~
f\in H^2(\triangle_H)
\end{equation*}
(see \cite[Section~3]{M2021}, \cite[Section~6]{GGLV2021}).
It follows that
\begin{align*}
H^2(\triangle_H)
=
\left\{
f(z)
=
\sum_{\alpha\in\mathcal I}
a_\alpha z^\alpha
:
a_\alpha\in\mathbb C,\;z\in \triangle_H,\;
\sum_{\alpha\in\mathcal I}|a_\alpha|^2<\infty
\right\},
\end{align*}
where $\mathcal I=\{\alpha=(\alpha_1, \alpha_2) \in \mathbb Z^2: \alpha_1\geq 0, ~\alpha_1+\alpha_2+1 \geq 0\}.$ Let $\mathcal Z$ be a complex separable Hilbert space. The $\mathcal Z$-valued Hardy space of the Hartogs triangle is defined by
\begin{align} \label{vectorhardy}
H^2_{\mathcal Z}(\triangle_H)
=
\left\{
f(z)
=
\sum_{\alpha\in\mathcal I}
a_\alpha z^\alpha
:
a_\alpha\in\mathcal Z,\; z\in \triangle_H,\;
\sum_{\alpha\in\mathcal I}\|a_\alpha\|_{\mathcal Z}^2<\infty
\right\}.
\end{align}
Equipped with the natural inner product,
$H^2_{\mathcal Z}(\triangle_H)$ is isometrically isomorphic to $H^2(\triangle_H)\otimes\mathcal Z.$ It is easy to see that $$H^2_{\mathcal Z}(\triangle_H)=\bigoplus_{\alpha \in \mathcal I}z^\alpha \mathcal Z.$$
Let $\mathscr M_{z_1}^{\mathcal Z}$ and $\mathscr M_{z_2}^{\mathcal Z}$ denote the multiplication operators by the coordinate functions $z_1$ and $z_2$, respectively, on $H^2_{\mathcal Z}(\triangle_H)$. It is easy to verify that both $\mathscr M_{z_1}^{\mathcal Z}$ and $\mathscr M_{z_2}^{\mathcal Z}$ are pure isometries. Recall that a pair of commuting operators $(V_1,V_2)$ and
$(W_1,W_2)$ on Hilbert spaces $\mathcal{H}$ and $\mathcal{K}$,
respectively, is said to be unitarily equivalent if there
exists a unitary map $U:\mathcal{H}\to\mathcal{K}$ such that
$UV_i=W_iU$ for $i=1,2$.
\begin{proof}[Proof of Proposition \ref{modelthm}]
   Since $H(V_1,V_2)\geq 0$ and both $V_1$ and $V_2$ are pure, Theorem~\ref{Hartogs bidisc core relation} yields a pure isometry $V_3$ commuting with $V_2$ such that
\begin{align*}
V_1=V_2V_3
\quad\text{and}\quad
C(V_3,V_2)\geq 0.
\end{align*}
By \cite[Lemma~6.2]{MSS2019}, the pair $(V_3,V_2)$ is doubly commuting. Hence, by \cite[Theorem~1]{Slocinski},
\begin{align*}
\mathcal H=\bigoplus_{\alpha_1, \alpha_2\in\mathbb Z_+}V_3^{\alpha_1}V_2^{\alpha_2}\mathcal Z,
\end{align*}
where
\begin{align*}
\mathcal Z=\mathcal W(V_2)\cap\mathcal W(V_3).
\end{align*}

Define
\begin{align*}
U:\bigoplus_{\alpha_1, \alpha_2\in\mathbb Z_+}V_3^{\alpha_1}V_2^{\alpha_2}\mathcal Z
\longrightarrow
H_{\mathcal Z}^2(\triangle_H)
=\bigoplus_{\alpha\in\mathcal I}z^\alpha\mathcal Z
\end{align*}
by
\begin{align*}
U\bigl(V_3^{\alpha_1}V_2^{\alpha_2}\gamma\bigr)
=
z_1^{\alpha_1}z_2^{\alpha_2-\alpha_1-1}\gamma,
\qquad
\gamma\in\mathcal Z,\;
\alpha_1, \alpha_2\in\mathbb Z_+.
\end{align*}
It is easy to see that $U$ extends to be a unitary operator. Moreover,
\begin{align*}
UV_1\bigl(V_3^{\alpha_1}V_2^{\alpha_2}\gamma\bigr)
&=
U\bigl(V_3^{\alpha_1+1}V_2^{\alpha_2+1}\gamma\bigr) \\
&=
z_1^{\alpha_1+1}z_2^{\alpha_2-\alpha_1-1}\gamma \\
&=
\mathscr M_{z_1}^\mathcal Z
U\bigl(V_3^{\alpha_1}V_2^{\alpha_2}\gamma\bigr),
\end{align*}
and similarly,
\begin{align*}
UV_2\bigl(V_3^{\alpha_1}V_2^{\alpha_2}\gamma\bigr)
&=
U\bigl(V_3^{\alpha_1}V_2^{\alpha_2+1}\gamma\bigr) \\
&=
z_1^{\alpha_1}z_2^{\alpha_2-\alpha_1}\gamma \\
&=
\mathscr M_{z_2}^\mathcal Z
U\bigl(V_3^{\alpha_1}V_2^{\alpha_2}\gamma\bigr)
\end{align*}
for each $\alpha_1, \alpha_2 \in \mathbb Z_+$ and $\gamma \in \mathcal Z.$ Therefore,
\begin{align*}
UV_1=\mathscr M_{z_1}^\mathcal ZU
\qquad\text{and}\qquad
UV_2=\mathscr M_{z_2}^\mathcal ZU,
\end{align*}
showing that $(V_1,V_2)$ is unitarily equivalent to $(\mathscr M_{z_1}^\mathcal Z,\mathscr M_{z_2}^\mathcal Z)$ on $H_{\mathcal Z}^2(\triangle_H)$. This completes the proof.
\end{proof}
Let $\{0\}\neq \mathcal M$ be a closed subspace of $H^2(\triangle_H)$ such that $\mathscr M_{z_i}\mathcal M \subseteq \mathcal M$ for $i=1,2.$ Then $(R_1, R_2)$ is a pair of pure commuting isometries on $\mathcal M,$ where $R_i=\mathscr M_{z_i}|_{\mathcal M}$ for $i=1,2.$ By Theorem \ref{modelthm}, the operator $H(R_1, R_2)\geq 0$ if and only if $(R_1, R_2)$ is unitarily equivalent to $(\mathscr M_{z_1}, \mathscr M_{z_2})$ on $H^2(\triangle_H).$ Indeed, $\mathcal Z=\mathcal W(R_2^*R_1) \cap \mathcal{W}(R_2)$ is of dimension $1.$

Similar to the Theorem \ref{Hartogs bidisc core relation}, we have the following result.
\begin{theorem} \label{Hartogs bidisc core relation1}
    Let $(V_1, V_2)$ be a pair of commuting isometries on a Hilbert space $\mathcal H.$ Then the following are equivalent:
    \begin{itemize}
       \item[$\mathrm{(i)}$]  $H(V_1, V_2) \leq 0.$
        \item[$\mathrm{(ii)}$] There exists an isometry $V_3$ commuting with $V_2$ such that $V_1=V_2V_3$ and $C(V_3, V_2) \leq 0.$
        \item[$\mathrm{(iii)}$]  $V_2(\mathcal{W}_{0,1}) \subseteq V_1(\mathcal{W}_{0,1}).$
    \end{itemize} 
\end{theorem}
\begin{proof}
$\mathrm{(i)} \iff \mathrm{(ii)} : $ Assume that $H(V_1, V_2) \leq 0.$ In view of \eqref{kernelw12}, for $x \in \mathcal{W}_{0,1},$ it is easy to see that
\beqn
\langle H(V_1, V_2)V_2x, V_2x \rangle = \|x\|^2 -\|V_1^*V_2x\|^2+\|V_1^*x\|^2.
\eeqn
Since $H(V_1, V_2) \leq 0,$ we have $V_1^*x=0.$ Now arguing similarly as in the proof Theorem \ref{Hartogs bidisc core relation} yields $\mathrm{(ii)}.$ For the converse part, observe that $H(V_1, V_2) = V_2C(V_3, V_2)V_2^*.$ Hence,  $H(V_1, V_2) \leq 0.$

$\mathrm{(i)} \iff \mathrm{(iii)} :$ It follows from the Lemma \ref{C-H relation}$\mathrm{(ii)}$.
\end{proof}
An application of Theorem \ref{Hartogs bidisc core relation} yields the following.
\begin{corollary}
    Let $(V_1, V_2)$ be a pair of doubly commuting isometries on a Hilbert space $\mathcal H$ such that $H(V_1, V_2) \geq 0.$ Then $V_2$ is unitary.
\end{corollary}
\begin{proof}
   Since $H(V_1,V_2)\geq 0,$ by Theorem \ref{Hartogs bidisc core relation}, we get $V_1\mathcal{W}_{0,1} \subseteq V_2\mathcal{W}_{0,1}.$ This combined with the fact that $(V_1, V_2)$ is doubly commuting yields $\mathcal {W}_{0,1}=\{0\}.$ Therefore, $V_2$ is unitary.
\end{proof}
In the previous result, if the assumption that  $(V_1, V_2)$ being doubly commuting is dropped, then $H(V_1, V_2) \geq 0$ need not imply that $V_2$ is unitary. It is worth mentioning that there exist pairs of commuting isometries with no unitary component in $V_2$ for which $H(V_1, V_2)\gneq 0.$ This phenomenon will be examined in the next section.

We now show that a negative or positive $\triangle_H$-core operator is zero when $\mathcal{W}_{0, 1}$ or $\mathcal{W}_{1, 0}$ are finite-dimensional.
\begin{corollary} \label{wanderingfinite}
     Let $(V_1, V_2)$ be a pair of commuting isometries on a Hilbert space $\mathcal H$ such that the subspaces $\mathcal{W}_{0, 1}$ or $\mathcal{W}_{1, 0}$ are finite-dimensional. Then $H(V_1, V_2) \geq 0$ or $H(V_1, V_2) \leq 0$ if and only if $H(V_1, V_2) = 0.$
\end{corollary}
\begin{proof}
 Assume that $\mathcal{W}_{0, 1}$ is finite-dimensional. Then by Theorems \ref{Hartogs bidisc core relation} and \ref{Hartogs bidisc core relation1}, if $H(V_1, V_2) \geq 0$ or $H(V_1, V_2) \leq 0,$ respectively, then $V_1(\mathcal{W}_{0, 1}) \subseteq V_2(\mathcal{W}_{0, 1})$ or $V_2(\mathcal{W}_{0, 1}) \subseteq V_1(\mathcal{W}_{0, 1}),$ respectively. Since $V_1$ and $V_2$ are isometries and $\operatorname{dim}\mathcal{W}_{0, 1} < \infty,$ we have 
     $$V_1(\mathcal{W}_{0, 1}) = V_2(\mathcal{W}_{0, 1}).$$ Thus, by Lemma \ref{C-H relation}(ii), we have  $H(V_1, V_2)=P_{V_2\mathcal{W}_{0, 1}}-P_{V_1\mathcal{W}_{0, 1}}=0.$
     Now assume that $\operatorname{dim}\mathcal{W}_{1, 0} < \infty$ along with $H(V_1, V_2) \geq 0$ or $H(V_1, V_2) \leq 0.$ Then from the proofs of Theorems \ref{Hartogs bidisc core relation} and \ref{Hartogs bidisc core relation1}, it follows that $\operatorname{dim}\mathcal{W}_{0, 1}< \infty.$ This completes the proof.
\end{proof}
We now establish a characterization of the positivity of
$H(V_1V_2, V_2)$ in terms of the invariance of wandering subspaces,
and deduce that $(V_1, V_2)$ is doubly commuting if and only if
$H(V_1V_2, V_2)\geq 0$. For this, we first recall the construction of the classical Wold-von Neumann decomposition of isometric operators on Hilbert spaces (see \cite{Wold}).
\begin{theorem}
If $V$ is an isometry on a Hilbert space $\mathcal{H}$ and $\mathcal{H}_{u}(V)=\displaystyle \bigcap_{n \in \mathbb Z_+}V^n\mathcal{H},$ then $\mathcal{H}_u(V)$ reduces $V,$ $V|_{\mathcal{H}_u(V)}$ is unitary, and $V|_{\mathcal{H}_s(V)}$ is a unilateral shift, where $\displaystyle \mathcal{H}_s(V)=\mathcal{H}_u(V)^{\perp}=\bigoplus_{n \in \mathbb Z_+}V^n\mathcal{W}(V),$ which also reduces $V.$
\end{theorem} As an application, we have the following.
\begin{lemma} \label{Doubly commuting lemma }
   Let $(V_1, V_2)$ be a pair of commuting isometries on a Hilbert space $\mathcal{H}.$ Then for all $j=1,2$ and $m,n \in \mathbb N,$
   \begin{itemize}
       \item[$\mathrm{(i)}$] $\mathcal{H}_s(V_{m, n})$ and $\mathcal{H}_u(V_{m, n})$ are $V_j$-reducing subspaces,

       \item[$\mathrm{(ii)}$] $\mathcal{H}_s(V_j) \subseteq \mathcal{H}_s(V_{m, n})$ and $ \mathcal{H}_u(V_{m, n}) \subseteq \mathcal{H}_u(V_j).$

   \end{itemize}
\end{lemma}
\begin{proof}
$\mathrm{(i)} :$ We first prove that $\mathcal{H}_s(V_{m, n})$ and $\mathcal{H}_u(V_{m, n})$ are $V_1$-reducing subspaces. Since $\mathcal{W}_{m-1, n}, \mathcal{W}_{1, 0} \subseteq \mathcal{W}_{m, n},$ by Lemma~\ref{wandering subspaces relation}(ii), it follows that $$V_1 \mathcal{W}_{m,n} = V_1(\mathcal{W}_{m-1, n}) \oplus V_{m,n}\mathcal{W}_{1,0} \subseteq \mathcal{W}_{m, n} \oplus V_{m,n}\mathcal{W}_{m, n} .$$ Consequently, for all $k \geq 0,$
\beqn
V_1V^k_{m,n}\mathcal{W}_{m,n} = V^k_{m,n}(V_1 \mathcal{W}_{m,n}) \subseteq V^k_{m,n}(\mathcal{W}_{m, n} \oplus V_{m,n}\mathcal{W}_{m, n}) \subseteq \mathcal{H}_s(V_{m,n}).
\eeqn
Thus we have, $V_1\mathcal{H}_s(V_{m, n}) \subseteq \mathcal{H}_s(V_{m, n}).$ Furthermore, by Lemma \ref{wandering subspaces relation}(iii),
\beqn
V_1^*V_{m,n}^k\mathcal{W}_{m,n} = V_{m,n}^{k-1}V_2\mathcal{W}_{m,n} &\subseteq& V_{m,n}^{k-1}V_2(\mathcal{W}_{m,n-1}+V_{m, n-1}\mathcal{W}_{0,1}) \\ &\subseteq& V_{m,n}^{k-1}(\mathcal{W}_{m,n} +V_{m,n}\mathcal{W}_{m,n})
\eeqn
for each $k \in \mathbb N.$ This combined with the fact that $V_1^*\mathcal{W}_{m,n} \subseteq \mathcal{W}_{m,n}$ yields $V_1^*\mathcal{H}_s(V_{m, n}) \subseteq \mathcal{H}_s(V_{m, n}).$ Hence, $\mathcal{H}_s(V_{m, n})$ is $V_1$-reducing subspace. Since $\mathcal{H}_u(V_{m, n}) = \mathcal{H}_s(V_{m, n})^\perp,$ it is easy to see that $\mathcal{H}_u(V_{m, n})$ is also $V_1$-reducing. Similarly, $\mathcal{H}_s(V_{m, n})$ and $\mathcal{H}_u(V_{m, n})$ are $V_2$-reducing subspaces.

$\mathrm{(ii)} :$ Given that $ V_{m, n}^k \mathcal{H}= V_{1}^k V_{m-1, n}^k \mathcal{H}= V_{2}^k V_{m, n-1}^k \mathcal{H} \subseteq V_{1}^k \mathcal{H},~  V_{2}^k \mathcal{H}$ for all $k \geq 0,$ we deduce that $\mathcal{H}_u(V_{m, n}) \subseteq \mathcal{H}_u(V_{j})$ for $j=1, 2.$ Additionally, the inclusion $\mathcal{H}_s(V_j) \subseteq \mathcal{H}_s(V_{m, n})$ follows from the observation that 
$$V_1^{*k}h\rightarrow 0, \mbox{ or }V_2^{*k}h\rightarrow 0 \implies V_{m,n }^{*k}h\rightarrow 0 \mbox{ as } k \rightarrow \infty$$ for any $h \in \mathcal H,$ $m,n \in \mathbb N$ and $j=1,2$. This completes the proof.
\end{proof}
The following results yield a characterization of $H(V_1V_2, V_2)\geq 0.$
\begin{proposition} \label{doublychara}
Assume that $m,n \in \mathbb N.$ Let $(V_1, V_2)$ be a pair of commuting isometries on a Hilbert space $\mathcal H.$ Then $H(V_1V_2, V_2)\geq 0$ if and only if $V_{0,n}(\mathcal{W}_{m,0}) \subseteq \mathcal{W}_{m, 0}.$   
\end{proposition}
\begin{proof}
    Assume that $H(V_1V_2, V_2)\geq 0.$ Then from Theorem \ref{Hartogs bidisc core relation} and \cite[Lemma~6.2]{MSS2019}, $V_1$ and $V_2$ are doubly commuting. Then for any $m, n \in \mathbb N,$ and $w \in \mathcal{W}_{m, 0},$
    $$V_1^{*m}V_2^nw = V_2^nV_1^{*m}w=0.$$
    Conversely, let $V_{0,n}(\mathcal{W}_{m,0}) \subseteq \mathcal{W}_{m, 0}$  for some $m,n \in \mathbb N.$ Let $\mathcal{H}=\mathcal{H}_u(V_{m, n}) \oplus \mathcal{H}_s(V_{m, n})$ be the Wold-von Neumann decomposition of $V_{m, n}=V_1^mV_2^n.$ By Lemma \ref{Doubly commuting lemma }(i), $\mathcal{H}_u(V_{m, n})$ and $\mathcal{H}_s(V_{m, n})$ are $V_1$ and $V_2$ reducing subspaces and $V_j|_{\mathcal{H}_u(V_{m, n})},$ $j=1, 2$ are unitary operators. Consequently, $V_1|_{\mathcal{H}_u(V_{m, n})}$ and $V_2|_{\mathcal{H}_u(V_{m, n})}$ are doubly commuting. 
    
    Now we claim that $V_1^*V_2=V_2V_1^*$ on $\mathcal{H}_s(V_{m, n}).$ Note that for any $k \in \mathbb N,$
    \beqn
(V_1^*V_2-V_2V_1^*)V_{m, n}^k&=&(V_1^*V_2-V_2V_1^*)V_1^{mk}V_2^{nk} \\ &=& V_1^{mk-1}V_2^{nk+1}-V_2^{nk+1}V_1^{mk-1}=0.
    \eeqn
    Hence, $V_1^*V_2=V_2V_1^*$ on $V_{m, n}^k\mathcal{W}_{m, n}$ for all $k \in \mathbb N.$ It now remains to show that $V_1^*V_2=V_2V_1^*$ on $\mathcal{W}_{m, n}.$ To show this, it suffices to show that $V_1^{*m}V_2^n=V_2^nV_1^{*m}$ on $\mathcal{W}_{m, n}.$ By Lemma \ref{wandering subspaces relation}, we have $\mathcal{W}_{m, n}=\mathcal{W}_{m, 0} \oplus V_{m, 0}\mathcal{W}_{0, n}.$ Let $w=w_1+V_1^m w_2 \in \mathcal{W}_{m, n}$ for some $w_1 \in \mathcal{W}_{m, 0}$ and $w_2 \in \mathcal{W}_{0, n}.$ Then
    \beqn
(V_1^{*m}V_2^n-V_2^nV_1^{*m})(w_1+V_1^m w_2)&=&V_1^{*m}V_2^n w_1 +V_1^{*m}V_2^n V_1^m w_2 \\ &-& V_2^nV_1^{*m}w_1- V_2^nV_1^{*m} V_1^m w_2, 
    \eeqn
    equivalently,
      \beqn
(V_1^{*m}V_2^n-V_2^nV_1^{*m})(w_1+V_1^m w_2)=0
    \eeqn
    as $V_{0, n} \mathcal{W}_{m, 0} \subseteq \mathcal{W}_{m, 0}.$ This completes the proof.
\end{proof}
\subsection{Definiteness of H-C}
 In this subsection, we study the definiteness of difference of the Hartogs triangle core operator and bidisc core operator. We first introduce two single operators which are related to the pair $(V_1, V_2)$ of commuting isometries (cf. \cite[p.~3]{HQY2015}). We define $J_1 \in  B(\mathcal W_{1,1})$ and $J_2\in  B(\mathcal W_{0,1})$ such that \begin{align*}
    J_1=P|_{\mathcal W_{1,1}}V_2|_{\mathcal W_{1,1}} \text{ and }  J_2=P|_{\mathcal W_{0,1}}V_1V_2|_{\mathcal W_{0,1}}.
\end{align*}

\begin{proposition} \label{frinfeisometry}
   The operator $J_1$ on $\mathcal W_{1, 1}$ is an isometry if and only if $V_2\mathcal W_{1, 1} \subseteq \mathcal W_{1, 1}.$ 
\end{proposition}
\begin{proof}
    Since $J_1^*J_1=P|_{\mathcal W_{1,1}}V_2^*|_{\mathcal W_{1,1}}P|_{\mathcal W_{1,1}}V_2|_{\mathcal W_{1,1}},$ we have 
    \begin{align} \label{I-j1j1}
       I_{\mathcal W_{1,1}}-J_1^*J_1=I_{\mathcal W_{1,1}}-P|_{\mathcal W_{1,1}}V_2^*|_{\mathcal W_{1,1}}P|_{\mathcal W_{1,1}}V_2|_{\mathcal W_{1,1}}.
    \end{align}
Using $P_{\mathcal W_{1,1}} + P_{\mathcal W_{1,1}^\perp} = I$ and the fact that $V_2$ is an isometry, we obtain
\begin{align*}
I_{\mathcal W_{1,1}}
&= P_{\mathcal W_{1,1}} V_2^* V_2|_{\mathcal W_{1,1}} \\
&= P_{\mathcal W_{1,1}} V_2^* P_{\mathcal W_{1,1}} V_2|_{\mathcal W_{1,1}}
+ P_{\mathcal W_{1,1}} V_2^* P_{\mathcal W_{1,1}^\perp} V_2|_{\mathcal W_{1,1}}.
\end{align*}
This combined with \eqref{I-j1j1} yields that
$$
I_{\mathcal W_{1,1}} - J_1^*J_1
=
P_{\mathcal W_{1,1}} V_2^* P_{\mathcal W_{1,1}^\perp} V_2|_{\mathcal W_{1,1}}=( P_{\mathcal W_{1,1}^\perp} V_2|_{\mathcal W_{1,1}})^* P_{\mathcal W_{1,1}^\perp} V_2|_{\mathcal W_{1,1}}.
$$
We now obtain that $J_1$ is an isometry if and only if $P_{\mathcal W_{1,1}^\perp} V_2|_{\mathcal W_{1,1}}=0.$ This completes the proof.
\end{proof}
Similary, one can show that the operator $J_2$ on $\mathcal W_{0, 1}$ is an isometry if and only if $V_1V_2\mathcal W_{0, 1} \subseteq \mathcal W_{0, 1}.$ Note that $(V_1V_2, V_2)$ is doubly commuting if and only if $V_2$ is unitary. This combined with Remark \ref{H-C relation}, Proposition \ref{frinfeisometry} and \cite[Lemma~6.2]{MSS2019} yields the following result.
\begin{theorem} \label{mainthmiso}
    Let $(V_1, V_2)$ be a pair of commuting isometries on $\mathcal H.$ Then the following are equivalent:
    \begin{enumerate}
    \item[$\mathrm{(i)}$] $H(V_1, V_2)-C(V_1, V_2) = 0.$
       \item [$\mathrm{(ii)}$]  $H(V_1, V_2)-C(V_1, V_2) \geq 0$ or $H(V_1, V_2)-C(V_1, V_2) \leq 0.$
        \item [$\mathrm{(iii)}$] $V_2$ is unitary.
        \item [$\mathrm{(iv)}$] $(V_1V_2, V_2)$ is doubly commuting.
        \item[$\mathrm{(v)}$] $H(V_1, V_2)-C(V_1, V_2)$ is a projection or negative of a projection.
        \item[$\mathrm{(vi)}$] The operator $J_1$ is an isometry.
        \item[$\mathrm{(vii)}$] The operator $J_2$ is an isometry.      
    \end{enumerate}
\end{theorem}
\section{On completely non-unitary isometric pairs via BCL theorem} \label{S4}
We now consider completely non-unitary isometric pairs. To this end, we compute the Hartogs triangle core operator for the BCL pair as in \eqref{BCLpairss2}. Although the calculations are simple, they are crucial for the results and analysis presented later in this section.
\begin{lemma} \label{BCLHartogs}
For the BCL pair $(M_{\phi_1}, M_{\phi_2})$ as in \eqref{BCLpairss2}, we have the following:
\begin{align} \label{Hmphi1mphi2}
H(M_{\phi_1}, M_{\phi_2})
&=P_{\mathbb C}\otimes (UP^\perp UPU^*P^\perp U^*-P)+
P_{\mathbb C}S_z^*\otimes UP^\perp UPU^*PU^* \notag
\\
&\quad+
S_zP_{\mathbb C}\otimes UPUPU^*P^\perp U^*
+
S_zP_{\mathbb C}S_z^*\otimes UPUPU^*PU^*.
\end{align}
\end{lemma}
\begin{proof}
Since $M_{\phi_1}^*
=
I\otimes UP
+
S_z^*\otimes UP^\perp,$ we have
\begin{align} \label{Mphi1}
M_{\phi_1}M_{\phi_1}^*
&=\left(I\otimes PU^*+S_z\otimes P^\perp U^*\right)
\left(I\otimes UP+S_z^*\otimes UP^\perp\right) \notag
\\
&= I\otimes P
+S_zS_z^*\otimes P^\perp \notag
\\
&= I-P_{\mathbb C}\otimes P^\perp.
\end{align}
Also, since $M_{\phi_2}^*
=
I\otimes P^\perp U^*
+
S_z^*\otimes PU^*,$ we get
\begin{align} \label{Mphi2}
M_{\phi_2}M_{\phi_2}^* \notag
&=
(I\otimes UP^\perp
+
S_z\otimes UP)
(I\otimes P^\perp U^*
+
S_z^*\otimes PU^*) \notag
\\
&=I\otimes UP^\perp P^\perp U^*
+S_z^*\otimes UP^\perp PU^*\notag
\\
&\quad
+S_z\otimes UPP^\perp U^*
+S_zS_z^*\otimes UPPU^*\notag
\\
&= I-P_{\mathbb C}\otimes UPU^*.
\end{align}
From \eqref{Mphi1} and $\eqref{Mphi2},$ we have 
\begin{align} \label{phi1phi2 diff}
   M_{\phi_1}M_{\phi_1}^*-M_{\phi_2}M_{\phi_2}^*
=
P_{\mathbb C}\otimes(-P^\perp+UPU^*). 
\end{align}
Using this, we have
\begin{align} \label{mixed difference}
-M_{\phi_2}^2M_{\phi_2}^{*2}+M_{\phi_1}M_{\phi_2}M_{\phi_1}^*M_{\phi_2}^*&=M_{\phi_2}[M_{\phi_1}M_{\phi_1}^{*}-M_{\phi_2}M_{\phi_2}^{*}]M_{\phi_2}^{*} \notag
\\
&=
\,P_{\mathbb C}\otimes (UP^\perp UPU^*P^\perp U^*-UP^\perp U^*) \notag
\\& \quad+
P_{\mathbb C}S_z^*\otimes UP^\perp UPU^*PU^* \notag
\\
&\quad+
S_zP_{\mathbb C}\otimes UPUPU^*P^\perp U^* \notag
\\& \quad+
S_zP_{\mathbb C}S_z^*\otimes UPUPU^*PU^* .
\end{align}

By using \eqref{phi1phi2 diff} and \eqref{mixed difference}, we have \eqref{Hmphi1mphi2}. This completes the proof.
\end{proof}
For the future use, we introduce the following notations:
\begin{align} \label{ABCD}
A&:=UP^\perp UPU^*P^\perp U^*-P,
\quad B:=UP^\perp UPU^*PU^*,\notag \\
C&:=UPUPU^*P^\perp U^*, \quad
D:=UPUPU^*PU^*,
\end{align}
where $\mathcal S,U$ and $P$ as in BCL triple. The following result characterizes when the $\triangle_H$-core operator of a BCL pair vanishes.
\begin{theorem} \label{pure=0}
    $H(M_{\phi_1},M_{\phi_2})=0$ if and only if $PUP=0$ and $U^2PU^{*2}=P.$
\end{theorem}
\begin{proof}
By Lemma \ref{BCLHartogs} and \eqref{ABCD}, we have
\begin{align*}
H(M_{\phi_1},M_{\phi_2})
=
P_{\mathbb C}\otimes A
+
P_{\mathbb C}S_z^*\otimes B
+
S_zP_{\mathbb C}\otimes C
+
S_zP_{\mathbb C}S_z^*\otimes D.
\end{align*}
It is easy to see that 
$H(M_{\phi_1},M_{\phi_2})=0$ on $z^2H^2(\mathbb T) \otimes \mathcal S.$ Therefore, it is enough to consider $H(M_{\phi_1},M_{\phi_2})$ on $\operatorname{span}\{1, z\}\otimes \mathcal S.$ It is easy to see that 
$$
\Bigl\{
P_{\mathbb C},
\,
P_{\mathbb C}S_z^*,
\,
S_zP_{\mathbb C},
\,
S_zP_{\mathbb C}S_z^*
\Bigr\}
$$
are linearly independent operators on the subspace $\{1,z\}.$
Hence
\begin{align}\label{Hiff0}
    H(M_{\phi_1}, M_{\phi_2})=0
\iff
A=B=C=D=0.
\end{align}
In particular, $D=UPUPU^*PU^*=0$ yields that $PUPU^*P=0.$ Since $PUPU^*P=(PUP)(PUP)^*,$ we have
\begin{align} \label{PUP0}
   PUP=0. 
\end{align}
This further yields that $$B=UP^\perp UPU^*PU^*=0
\mbox{ and }C=UPUPU^*P^\perp U^*=0.$$
Now, by \eqref{PUP0}
\begin{align} \label{Pupup}
 P^\perp UP=UP.   
\end{align}
By \eqref{PUP0} and \eqref{Pupup}, we have 
\begin{align*}
U(P^\perp UP)U^*P^\perp U^* =UUPU^*P^\perp U^*
=UU(P^\perp UP)^* U^*
=U^2PU^{*2}.
\end{align*}
This combined with \eqref{ABCD} gives that $A=U^2PU^{*2}-P.$ It now follows from \eqref{Hiff0} and \eqref{PUP0} that 
$$
H(M_{\phi_1}, M_{\phi_2})=0
\iff
\begin{cases}
PUP=0,\\
U^2PU^{*2}=P.
\end{cases}
$$
This completes the proof.
\end{proof}
It is worth mentioning that $PUP=0$ and $U^2PU^{*2}=P$ may not yield that $P=0.$ For instance, $$
\mathcal S=\mathbb C^2,
\qquad
U=
\begin{pmatrix}
0 & 1\\
1 & 0
\end{pmatrix},
\qquad
P=
\begin{pmatrix}
1 & 0\\
0 & 0
\end{pmatrix}.
$$
\begin{remark}
Unlike the case of bidisc core operator (see \cite[Proposition~3.1]{HQY2015}), the commutativity of $U$ and $P$ may not yield $H(M_{\phi_1}, M_{\phi_2})=0$ unless $P=0.$ Also, $H(M_{\phi_1}, M_{\phi_2})$ is indefinite whenever $P\neq 0$ and $UP=PU.$
\end{remark}
The following result concerns definiteness of $H(M_{\phi_1},M_{\phi_2}).$
\begin{proposition} \label{H>0}
    $H(M_{\phi_1},M_{\phi_2})\geq0$ if and only if $A\geq 0, D \geq0$ and $|\langle Bs,s\rangle|^2
    \le
    \langle As,s\rangle \langle Ds,s\rangle$ for $s \in \mathcal S,$ where $A,B,D$ as in \eqref{ABCD}. 
\end{proposition}
\begin{proof}
    Note that $H(M_{\phi_1},M_{\phi_2})=0$ on $z^2H^2(\mathbb T) \otimes \mathcal S.$ To check positivity of $H(M_{\phi_1},M_{\phi_2}),$ it is enough to consider it on $\operatorname{span}\{1, z\}\otimes \mathcal S.$ For $\tilde z=(\alpha +\beta z) \otimes s,$ where $\alpha,\beta \in \mathbb C$ and $s \in \mathcal S,$
  \begin{align} \label{positivityHBCL}
      \langle H(M_{\phi_1},M_{\phi_2})(\tilde z),\tilde z\rangle=|\alpha|^2\langle As,s \rangle+2\operatorname{Re}(\beta\bar\alpha \langle Bs,s \rangle)+|\beta|^2\langle Ds,s \rangle,
    \end{align}
    where we used the fact that $B=C^*.$ Then \eqref{positivityHBCL} becomes
$$
\langle H(M_{\phi_1},M_{\phi_2})(\tilde z),\tilde z\rangle=\begin{pmatrix}\alpha & \beta\end{pmatrix}
\begin{pmatrix}
\langle As,s \rangle & \langle Bs,s \rangle\\
\overline{\langle Bs,s \rangle} & \langle Ds,s \rangle
\end{pmatrix}
\begin{pmatrix}
\bar\alpha\\
\bar\beta
\end{pmatrix}
$$
for all $\alpha,\beta\in\mathbb C$.
Hence, $H(M_{\phi_1},M_{\phi_2})\geq0$ if and only if \begin{align*}
\langle As,s \rangle\geq0, \quad\langle Ds,s \rangle\geq0 \mbox{ and }   |\langle Bs,s\rangle|^2
    \le
    \langle As,s\rangle \langle Ds,s\rangle
\end{align*}
for each $s \in \mathcal S.$ This completes the proof.
\end{proof}
Similarly, we have the following.
\begin{proposition} \label{H<0}
    $H(M_{\phi_1},M_{\phi_2})\leq0$ if and only if $A\leq 0, D \leq0$ and $|\langle Bs,s\rangle|^2
    \le
    \langle As,s\rangle \langle Ds,s\rangle$ for $s \in S,$ where $A,B,D$ as in \eqref{ABCD}.
\end{proposition}
\begin{remark} \label{A=0rema}
For any BCL triple $(\mathcal S,U,P),$ if $D\leq0$ then $D=0$ (see \eqref{ABCD}). Also, $H(M_{\phi_1},M_{\phi_2})\neq0$ if and only if $A$ is a nonzero operator. Indeed, $$A=0 \implies PUP=0.$$ Thus $B=C=D=0$ (see \eqref{ABCD}). Hence, $H(M_{\phi_1},M_{\phi_2})=0.$ On the other hand, $PUP=0$ indeed gives us  $B=C=D=0.$ However, it may not imply that $A=0.$
\end{remark}
The following result shows that $A=0$ when $\mathcal S$ is finite-dimensional and $A\geq0.$
\begin{lemma} \label{A>0impliessA=0}
Assume that $\mathcal S$ is a finite-dimensional space in the BCL triple $(\mathcal S, U, P).$  If $A\geq0$ then $A=0,$ where $A$ is as in \eqref{ABCD}. 
\end{lemma}

\begin{proof}
Set $M=A+P.$ Since $M-P\geq 0$, we have \begin{align} \label{Trace>0}
    \operatorname{tr}(M-P)\geq 0,
\end{align}
where $\operatorname{tr}$ denotes the trace. Set $Q=UP^\perp U.$ Then
$$
\operatorname{tr}(M)=\operatorname{tr}(UP^\perp UPU^*P^\perp U^*)=\operatorname{tr}(QPQ^*).
$$
Now the cyclicity of trace gives
$$
\operatorname{tr}(QPQ^*)
=
\operatorname{tr}(PQ^*Q)
= \operatorname{tr}(PQ^*QP).
$$
Hence $\operatorname{tr}(M)=\operatorname{tr}(PQ^*QP).$ Using the fact that $0\leq Q^*Q\leq I$, it follows that
$$
\operatorname{tr}(M)
=
\operatorname{tr}(PQ^*QP)
\leq
\operatorname{tr}(P).
$$
This combined with \eqref{Trace>0} yields that $
\operatorname{tr}(M)=\operatorname{tr}(P).$ Since $M$ and $P$ are projections on finite-dimensional space $\mathcal S$, we have $A=0.$
\end{proof}
In the case of bidisc core operator, $C(M_{\phi_1}, M_{\phi_2})$ is zero when $\mathcal S$ is finite-dimensional and $C(M_{\phi_1}, M_{\phi_2})\geq 0$ or $C(M_{\phi_1}, M_{\phi_2})\leq 0$ (see \cite[Proposition 3.6]{HQY2015}). Similar to this, in the case of Hartogs triangle core operator, we have the following.
\begin{corollary}
    Assume that $\mathcal S$ is a finite-dimensional space in the BCL triple $(\mathcal S, U, P)$. If $H(M_{\phi_1}, M_{\phi_2}) \geq 0$ (or $\leq 0$) then $H(M_{\phi_1}, M_{\phi_2}) = 0.$
\end{corollary}
\begin{proof}
Assume that $H(M_{\phi_1}, M_{\phi_2}) \geq 0.$ Then by Proposition \ref{H>0}, we have $A\geq 0.$ Now an application of Lemma \ref{A>0impliessA=0}, we have $A=0.$ Hence, Remark \ref{A=0rema} and \eqref{ABCD} yield that $H(M_{\phi_1}, M_{\phi_2}) = 0.$

Now assume that $H(M_{\phi_1}, M_{\phi_2}) \leq 0.$ Then by Proposition \ref{H<0}, we have $D\leq 0.$ Therefore, by Remark \ref{A=0rema}, we have $PUP=0.$ Hence $B=C=0.$ Since $\mathcal S$ is finite-dimensional and $U^2PU^{*2}$ is unitarily equivalent to $P$, it follows that $U^2PU^{*2}$ and $P$ have the same rank. This combined with the fact that $A\leq 0$ yields that $A=0$ (see  Proposition \ref{H<0}). This completes the proof.
\end{proof}
The following result characterizes when the core operator $H(M_{\phi_1},M_{\phi_2})$ is negative.
\begin{corollary} \label{Negative Core operator Characterization}
$H(M_{\phi_1},M_{\phi_2}) \leq 0$ if and only if $PUP=0$ and $P-U^2PU^{*2}$ is a projection.
\end{corollary}
\begin{proof}
Since $H(M_{\phi_1},M_{\phi_2}) \leq 0,$ by Proposition \ref{H<0} and Remark \ref{A=0rema}, we have $PUP=0.$ Thus, the operators $B,C$ and $D$ as in \eqref{ABCD} are zero. Let $
Z:=U^2PU^{*2}$, and $W:=P-Z.
$ Since $H(M_{\phi_1},M_{\phi_2})$ is negative, again by Proposition \ref{H>0}, we have $
W\ge 0$. Since $Z$ and $P$ are projections, we obtain
\begin{align} \label{Qprouse}
    ZW+WZ+W^2=W.
\end{align}
Therefore, $ZWZ+ZW^2Z=0.$ Since $W\geq 0,$ $ZWZ=0$ and $ZW^2Z=0.$ This further yields that $WZ=0$ and $ZW=0.$ This combined with \eqref{Qprouse} yields that $W^2=W.$ Hence, $W$ is a projection. The converse part is elementary.
\end{proof}
\begin{remark}
From the proof of Corollary \ref{Negative Core operator Characterization}, it is clear that if $PUP=0,$ then $H(M_{\phi_1},M_{\phi_2}) \geq 0$ if and only if $U^2PU^{*2}-P$ is a projection.    
\end{remark}
We end this section with a characterization for the compactness of the operator $H(M_{\phi_1},M_{\phi_2})$.
\begin{theorem} \label{Compactchara}
  $H(M_{\phi_1},M_{\phi_2})$ is a compact operator if and only if both $PUP$ and $U^2P-PU^2$ are compact.
\end{theorem}
\begin{proof}
    Assume that $H(M_{\phi_1},M_{\phi_2})$ is a compact operator. On the subspace $\operatorname{span}\{z\} \otimes \mathcal S$ of $H^2(\mathbb T)\otimes \mathcal S,$ the operator $H(M_{\phi_1},M_{\phi_2})$ becomes $P_{\mathbb C}S_z^*\otimes B
+
S_zP_{\mathbb C}S_z^*\otimes D.$ Since $H(M_{\phi_1},M_{\phi_2})$ is a compact operator, we have 
\begin{align*}
   T:=P_{\mathbb C}S_z^*\otimes B+
S_zP_{\mathbb C}S_z^*\otimes D: \operatorname{span}\{z\} \otimes \mathcal S \rightarrow H^2(\mathbb T)\otimes \mathcal S
\end{align*} 
is a compact operator.
For a bounded sequence $(s_n)_{n \in \mathbb N}$ in $\mathcal S,$ we have
\begin{align*}
\|T(z\otimes  s_n)-T(z\otimes s_m)\|^2
&=
\|1\otimes A(s_n-s_m)+z\otimes B(s_n-s_m)\|^2 \\
&=
\|A(s_n-s_m)\|^2+\|B(s_n-s_m)\|^2,
\end{align*}
where $m,n \in \mathbb N.$ The above identity combined with the compactness of $T$ gives that $B$ and $D$ are compact operators (see \eqref{ABCD}). In particular, $D=UPUPU^*PU^*$ is compact if and only if $ PUPU^*P$ is compact.
Hence, $PUP$ is compact. Similarly, on the subspace $\mathbb C \otimes \mathcal S$ of $H^2(\mathbb T)\otimes \mathcal S,$ the operator $H(M_{\phi_1},M_{\phi_2})$ becomes $
P_{\mathbb C}\otimes A
+
S_zP_{\mathbb C}\otimes C
.$ Since $H(M_{\phi_1},M_{\phi_2})$ is a compact operator, $A$ is a compact operator. Since $PUP$ is compact, the operator $P^\perp UP-UP$ is compact. Thus
\begin{align}\label{essential comm}
    UP^\perp UPU^*P^\perp U^*-P \notag &=U UPU^*P^\perp U^*-P+\mbox{compact}\\&=U^2PU^{*2}-P+\mbox{compact}.
\end{align}
Thus, the compactness of $A$ yields the compactness of $U^2P-PU^2$. Conversely, let both $PUP$ and $U^2P-PU^2$ are compact. Since $PUP$ is compact, it is easy to see that operators $B,C,D$ as in \eqref{ABCD} are compact. The compactness of $A$ follows from \eqref{essential comm}. This completes the proof. 
\end{proof}
\subsection{Structure of compact Hartogs triangle core operator}
We now completely characterize the structure of compact and contractive $\triangle_H$-core operators (cf. \cite[Theorem~4.3]{HQY2015}). For an eigenvalue $a$ of $H(M_{\phi_1},M_{\phi_2}),$ let $E(a)$ denote the corresponding eigenspace. Also, we denote $\operatorname{ker} M_{\phi_2}^*=\mathcal{W}_{0,1}$ and $\operatorname{ker} M_{\phi_1}^*=\mathcal{W}_{1,0}.$
\begin{lemma}
    For the Hartogs triangle core operator $H(M_{\phi_1},M_{\phi_2})$ acting on $H^2_\mathcal S(\mathbb T),$ 
\begin{align*}
 & E(1)=  \left\{
M_{\phi_2}y:
y\in\operatorname{ker} M_{\phi_2}^*,
\quad
M_{\phi_1}^*M_{\phi_2}y\in \operatorname{ran}M_{\phi_2}
\right\}\\&E(-1)= \left\{
M_{\phi_1}y:
y\in\operatorname{ker} M_{\phi_2}^*,
\quad
M_{\phi_2}^*M_{\phi_1}y\in \operatorname{ran}M_{\phi_2}
\right\}.
\end{align*}
\end{lemma}
\begin{proof}
Since
$$
H(M_{\phi_1},M_{\phi_2})=M_{\phi_2}M_{\phi_2}^*-M_{\phi_1}M_{\phi_1}^*-M_{\phi_2}^2M_{\phi_2}^{*2}+M_{\phi_1}M_{\phi_2}M_{\phi_1}^*M_{\phi_2}^* .
$$
let $x$ be the eigenvector corresponding to eigenvalue $1.$ Then by Lemma \ref{C-H relation}, we have
\begin{align*}
    \|x\|^2=\|P_{\mathcal W_{0,1}}M_{\phi_2}^*x\|^2-\|P_{\mathcal W_{0,1}}M_{\phi_1}^*x\|^2.
\end{align*}
This further yields that 
\begin{align*}
    \|P_{\mathcal W_{0,1}}M_{\phi_1}^*x\|^2=\|P_{\mathcal W_{0,1}}M_{\phi_2}^*x\|^2-\|x\|^2.
\end{align*}
Since $\|P_{\mathcal W_{0,1}}M_{\phi_2}^*x\|^2-\|x\|^2\leq 0,$ we have 
\begin{align*}
    \|P_{\mathcal W_{0,1}}M_{\phi_1}^*x\|^2=0 \mbox { and }\|P_{\mathcal W_{0,1}}M_{\phi_2}^*x\|^2=\|x\|^2.
\end{align*}
Hence
\begin{align*}
    E(1)=M_{\phi_2}(\operatorname{ker} M_{\phi_2}^*)
\cap
\{x\in H^2_\mathcal S(\mathbb T):\ M_{\phi_1}^*x\in \operatorname{ran}M_{\phi_2}\},
\end{align*}
equivalently,
\begin{align*}
   E(1)= \left\{
M_{\phi_2}y:
y\in\operatorname{ker} M_{\phi_2}^*,
\quad
M_{\phi_1}^*M_{\phi_2}y\in \operatorname{ran}M_{\phi_2}
\right\}.
\end{align*}
Along the similar idea, $E(-1)$ can be obtained.
\end{proof}
The following result concerns eigenvalues in the open interval $(-1,1)$. The method employed here is comparatively elementary and also provides a simple proof of \cite[Lemma~4.2]{HQY2015}.
\begin{lemma}
If $a\in (-1,1)\setminus\{0\}$ is an eigenvalue of Hartogs triangle core operator $H(M_{\phi_1},M_{\phi_2})$, then $-a$ is also an eigenvalue of $H(M_{\phi_1},M_{\phi_2})$. Moreover, if $E(a)$ is finite-dimensional, then $\operatorname{dim}E(a)=\operatorname{dim}E(-a).$
\end{lemma}

\begin{proof}
By Lemma \ref{C-H relation}, we have 
$$H(M_{\phi_1},M_{\phi_2})=P_{M_{\phi_2}\mathcal{W}_{0, 1}}-P_{M_{\phi_1}\mathcal{W}_{0, 1}}.$$
Now define
$$
X:=P_{M_{\phi_2}\mathcal{W}_{0, 1}}+P_{M_{\phi_1}\mathcal{W}_{0, 1}}-I .
$$
Since  $P_{M_{\phi_2}\mathcal{W}_{0, 1}}^2=P_{M_{\phi_2}\mathcal{W}_{0, 1}}$ and $P_{M_{\phi_1}\mathcal{W}_{0, 1}}^2=P_{M_{\phi_1}\mathcal{W}_{0, 1}}$, we have 
\begin{align} \label{Comutingxh}
 H(M_{\phi_1},M_{\phi_2})X=- XH(M_{\phi_1},M_{\phi_2}).
\end{align}
Now let $a\in (-1,1)$ be nonzero eigenvalue of $H(M_{\phi_1},M_{\phi_2})$ corresponding to eigenvector $x.$ Then by using \eqref{Comutingxh},
\begin{align*}
  H(M_{\phi_1},M_{\phi_2})Xx=-XH(M_{\phi_1},M_{\phi_2})x=-aXx.  
\end{align*}
We now claim that $Xx\neq 0.$ Assume, for contradiction, that $Xx=0$. Then
$$
(P_{M_{\phi_2}\mathcal{W}_{0, 1}}+P_{M_{\phi_1}\mathcal{W}_{0, 1}})x=x .
$$
This together with
$$
(P_{M_{\phi_2}\mathcal{W}_{0, 1}}-P_{M_{\phi_1}\mathcal{W}_{0, 1}})x=ax,
$$
yields that 
$$
P_{M_{\phi_2}\mathcal{W}_{0, 1}}x=\frac{1+a}{2}x,
\qquad
P_{M_{\phi_1}\mathcal{W}_{0, 1}}x=\frac{1-a}{2}x.
$$
Since $P_{M_{\phi_2}\mathcal{W}_{0, 1}}$ and $P_{M_{\phi_1}\mathcal{W}_{0, 1}}$ are orthogonal projections, their only possible eigenvalues are $0$ and $1$. Therefore $\frac{1\pm a}{2}\in\{0,1\},$ which implies $a=\pm1$, contradicting
$$
a\in(-1,1).
$$
Hence $Xx\neq0$. It also follows that $X:E(a)\rightarrow E(-a)$ is an injective map. Thus, if $E(a)$ is finite-dimensional, then applying the same argument with $-a$ in place of $a$ yields that $\operatorname{dim}E(a)=\operatorname{dim}E(-a).$ 
\end{proof}
If the Hartogs triangle core operator $H(M_{\phi_1},M_{\phi_2})$ is a compact contraction, then its nonzero spectra are eigenvalues in $[-1,1].$ Therefore the closure of range of $H(M_{\phi_1},M_{\phi_2})$ can be decomposed as
\begin{align*}
    E(1)\oplus \left(\bigoplus_{0< a_j<1}E(a_j)\right)\oplus E(-1)\oplus\left(\bigoplus_{-1< a_j<0}E(a_j)\right).
\end{align*}
For brevity, let $d_1=\operatorname{dim}E(1),d_{-1}=\operatorname{dim}E(-1),$ and $$E_{(0,1)}=\bigoplus_{0< a_j<1}a_jQ_j,$$
where $ Q_j$ is the orthogonal projection from $H^2_\mathcal S(\mathbb T)$ onto $E(a_j).$ The next theorem summarizes the preceeding observations (cf. \cite[Theorem 4.3]{HQY2015}) 
\begin{theorem} \label{structurethm}
    If $H(M_{\phi_1},M_{\phi_2})$ is a compact contraction, then its nonzero part is unitarily equivalent to the diagonal block matrix 
    $$
\begin{pmatrix}
I_{d_1} & 0 & 0 & 0\\
0 & E_{(0,1)} & 0 & 0\\
0 & 0 & -I_{d_{-1}} & 0\\
0 & 0 & 0 & -E_{(0,1)}
\end{pmatrix},
$$
where $I_n$ denotes the $n \times n$ identity matrix.
\end{theorem}
If $H(M_{\phi_1},M_{\phi_2})$ is compact and negative, then by Corollary \ref{Negative Core operator Characterization}, the only possible nonzero eigenvalue of $H(M_{\phi_1},M_{\phi_2})$ is $-1.$ Hence, we have the following result.
\begin{corollary}
    Let $H(M_{\phi_1},M_{\phi_2})$ be compact and negative operator. Then $H(M_{\phi_1},M_{\phi_2})$ is a finite rank operator.
\end{corollary}
We end this section with the following result which shows that the compactness and positivity of the core operator $H(M_{\phi_1},M_{\phi_2})$ also imply that $H(M_{\phi_1},M_{\phi_2})$ has finite rank.
\begin{proposition}
 Let core operator $H(M_{\phi_1},M_{\phi_2})$ be compact and positive. Then $H(M_{\phi_1},M_{\phi_2})$ is a finite rank operator.
\end{proposition}

\begin{proof}
Assume that $H(M_{\phi_1},M_{\phi_2})$ is compact and positive. It now follows from the proof of Corollary \ref{Negative Core operator Characterization} that $A$ is itself an orthogonal projection. Since $H(M_{\phi_1},M_{\phi_2})$ is compact, by Theorem \ref{Compactchara}, $A$ is a finite rank operator (see \eqref{ABCD}). Again from the proof of Theorem \ref{Compactchara}, it follows that $PUP$ is finite rank. Hence, by \eqref{ABCD}, $B,C$ and $D$ are also finite rank. Consequently, $H(M_{\phi_1},M_{\phi_2})$ is finite rank.
\end{proof}
\section{Concluding remarks}
In this work, we introduce the Hartogs triangle core operator associated with pairs of commuting isometries (see \eqref{Hartogs triangle core operator}) and establish its fundamental properties. We expect that this operator will provide a useful tool for investigating the structure of pairs of commuting isometries and the submodules of the Hardy space over the Hartogs triangle (for a recent study on submodules, see \cite{CGJ2026}). In subsequent work, we intend to develop these applications further, with particular emphasis on establishing a spectral dichotomy for pairs of commuting isometries via the $\triangle_H$-core operator.

\vspace{.5cm}
\noindent \textbf{Acknowledgments:}  The research of the author is supported
by a post-doctoral fellowship provided by the National Board for Higher Mathematics
(NBHM), India (Order No: 0204/37(2)/2024/NBHM/R\&D-II/16542, dated November 22, 2024).
{}

\end{document}